\documentclass[12pt,a4paper]{amsart}
\usepackage[top=1.15in, bottom=1.15in, left=1.17in, right=1.17in]{geometry}
\usepackage{amssymb}
\usepackage{pdflscape}
\usepackage{amscd}
\usepackage{fancyhdr}
\usepackage{tikz,ltablex,booktabs,tikz-cd}
\usepackage{graphicx}
\usepackage{color}
\usepackage[all]{xy}    % Activate to display a given date or no date
\usepackage{comment}

\newcommand{\qt}{\widetilde{Q}_0}

\newcommand{\act}{\textbf{a}}
\newcommand{\A}{\mathcal{A}}

\newcommand{\bC}{\mathbb{C}}

\newcommand{\tQ}{\widetilde{Q}}

\providecommand{\keywords}[1]
{
  \small	
  \textbf{\textit{Keywords---}} #1
}

\newcommand{\pre}{\operatorname{\textbf{pre-}}\nolimits}

\def\Hom{\mbox{Hom}}
\def\Ker{\mbox{Ker}}
\def\Im{\mbox{Im}}
\def\rep{\mbox{rep\,}}

\newtheorem{theorem}{Theorem}[section]
\newtheorem{corollary}[theorem]{Corollary}
\newtheorem{lemma}[theorem]{Lemma}
\newtheorem{proposition}[theorem]{Proposition}

\newtheorem{definition}[theorem]{Definition}
\newtheorem{remark}[theorem]{Remark}
\newtheorem{openquestion}[theorem]{Open Question}

\title{Double framed moduli spaces of quiver representations}
\author{Marco Armenta, Thomas Br\"{u}stle, Souheila Hassoun, Markus Reineke}
\begin{document}

\parindent0pt

\begin{abstract}
    Motivated by problems in the neural networks setting, we study moduli spaces of double framed quiver representations and give both a linear algebra description and a representation theoretic description of these moduli spaces. We define a network category whose isomorphism classes of objects correspond to the orbits of quiver representations, in which neural networks map input data. We then prove that the output of a neural network depends only on the corresponding point in the moduli space. Finally, we present a different perspective on mapping neural networks with a specific activation function, called ReLU, to a moduli space using the symplectic reduction approach to quiver moduli.
\end{abstract}



\maketitle

%\tableofcontents

\section{Introduction}

Moduli spaces of framed quiver representations have been considered since \cite{King94, Rei} and the dual theory (coframed quiver representations) has been also considered independently, see for example \cite{Rei}. Then, double framing (framing and co-framing at the same time) was found to describe neural networks and the computations they perform on data, see \cite{AJ20}. Motivated by these applications we perform an analysis of the underlying geometric and representation theoretic perspectives of such moduli spaces of double-framed quiver representations.

Let $Q$ be a finite acyclic quiver. Let $s_1,\ldots,s_p$ be the sources of $Q$, and let $t_1,\ldots,t_q$ be the sinks of $Q$. Let $\tilde{Q}$ be the full subquiver of $Q$ on the vertices $\qt:=Q_0\setminus\{s_1,\ldots,s_p,t_1,\ldots,t_q\}$. Let ${\bf d}$ be a dimension vector for $Q$, and fix complex vector spaces $V_i$ of dimension $d_i$ for all $i\in Q_0$. Let $$R_{\bf d}(Q)=\bigoplus_{\alpha:i\rightarrow j}{\rm Hom}(V_i,V_j)$$ be the variety of complex representations of $Q$ of dimension vector ${\bf d}$, and let $$G_{\bf d}(Q)=\prod_{i\in Q_0}{\rm GL}(V_i)$$ be the base change group, acting on $R_{\bf d}(Q)$ via
$$(g_i)_i\cdot(V_\alpha)_\alpha=(g_jV_\alpha g_i^{-1})_{\alpha:i\rightarrow j}.$$ We consider the subgroup $$G_{\bf d}(\tilde{Q})=\{(g_i)_i\in G_{\bf d}(Q)\, |\, g_i=1\mbox{ if }i\not\in \qt \}.$$
We will define and study moduli spaces parametrizing $G_{\bf d}(\tilde{Q})$-orbits in $R_{\bf d}(Q)$ and then relate this to the double framing setting via deframing. This allows us to study different stability conditions for such representations and link the different moduli spaces that arise, particularly in the thin case which is important for neural networks. We then give a representation theoretic description of these moduli spaces which we then translate to a purely linear algebra description. We use the representation theoretic approach to these moduli spaces to study the problem from a categorical point of view. Later on, we will present applications to the theory of neural networks as stated in \cite{AJ20}.

\section{Deframing}\label{section:deframing}

To interpret this problem as the problem of finding double framed versions of moduli spaces, we first make the following definition:

\begin{definition} For every vertex $i\in \qt$, define
$$U_i=\bigoplus_{\alpha:s_k\rightarrow i}V_{s_k},\, W_i=\bigoplus_{\alpha:i\rightarrow t_l}V_{t_l}.$$
\end{definition}

In other words, $U_i$ contains one copy of $V_s$ for every arrow from a source $s$ to $i$, and similarly $W_i$ contains one copy of $V_t$ for every arrow from $i$ to a sink $t$.  We denote $u_i=\dim U_i$ and $w_i=\dim W_i$. We already note at this point that $u_i\not=0$ for every source $i$ of $\tilde{Q}$ (otherwise there would be no arrow to $i$ in $Q$, that is, $i$ would have been a source in $Q$), and similarly $w_i\not=0$ for every sink of $i$.

The above definition allows us, for every $i\in \qt$, to collect all maps $V_\alpha:V_s\rightarrow V_i$ into a single map
$$f_i=(V_\alpha)_{\alpha:s_k\rightarrow i}:U_i\rightarrow V_i,$$
and similarly
$$h_i=(V_\alpha)_{\alpha:i\rightarrow t_l}:V_i\rightarrow W_i.$$
Denoting by ${\bf\tilde d}$ the restriction of ${\bf d}$ to $\tilde{Q}$, we thus find:

\begin{lemma}\label{l22} The above construction induces an isomorphism of affine spaces
$$R_{\bf d}(Q)\simeq R_{\bf\tilde{d}}(\tilde{Q})\times\bigoplus_{i\in \qt}{\rm Hom}(U_i,V_i)\times\bigoplus_{i\in \qt}{\rm Hom}(V_i,W_i),$$
and the action of $G_{\bf d}(\tilde{Q})$ translates to
$$(g_i)_{i\in \qt}\cdot((V_\alpha)_\alpha,(f_i)_i,(h_i)_i)=((g_jV_\alpha g_i^{-1})_{\alpha:i\rightarrow j},(g_if_i)_i,(h_ig_i^{-1})_i).$$
\end{lemma}

\begin{definition} We introduce the deframed quiver $Q'$: its set of vertices is $$Q'_0=\qt \cup\{\infty\},$$ and its set of arrows is 
$$Q'_1=\tilde{Q}_1\cup\{\beta_{i,k}:\infty\rightarrow i\, |\, i\in \qt,\, k=1,\ldots,u_i\}\cup$$
$$\cup\{\gamma_{i,l}:i\rightarrow \infty\,|\, i\in \qt,\, l=1,\ldots,w_i\}.$$
We define a dimension vector ${\bf d'}$ for $Q'$ by
$$d'_i=d_i\mbox{ for }i\in \qt,\; d'_\infty=1.$$
\end{definition}

A representation of $Q'$ of dimension vector ${\bf d'}$ is thus given by a representation $V$ of $\tilde{Q}$ of dimension vector ${\bf\tilde d}$, together with vectors $v_{i,k}\in V_i$ representing the arrows $\beta_{i,k}$, and covectors $\varphi_{i,l}\in V_i^*$ representing the arrows $\gamma_{i,l}$. The base change group $$G_{\bf d'}(Q')\simeq\mathbb{C}^*\times G_{\bf d}(\tilde{Q})$$
acts on such a representation by
$$(\lambda,(g_i)_i)\cdot((V_\alpha)_\alpha,(v_{i,k})_{i,k}),(\varphi_{i,l})_{i,l})=((g_jV_\alpha g_i^{-1})_{\alpha:i\rightarrow j},(\lambda^{-1}\cdot g_iv_{i,k})_{i,k},(\lambda\cdot\varphi_{i,l}g_i^{-1})_{i,l})).$$
Choose bases for all spaces $U_i$ and $W_i$ defined above. Then, for every $i\in \qt$, we can canonically collect the vectors $v_{i,k}$ into a map $f_i:U_i\rightarrow V_i$, and we can canonically collect the covectors $\varphi_{i,l}$ into a map $h_i:V_i\rightarrow W_i$. This shows (using the above identification):

\begin{lemma} We have
$$R_{\bf d'}(Q')\simeq R_{\bf d}(Q),$$
and this isomorphism is equivariant with respect to the action of $G_{\bf d}(\tilde{Q})$, considered as a subgroup of $G_{\bf d'}(Q')$.
\end{lemma}

We now claim:

\begin{lemma} The groups $G_{\bf d}(\tilde{Q})$ and $G_{\bf d'}(Q')=\mathbb{C}^*\times G_{\bf d}(\tilde{Q})$ have the same orbits in $R_{\bf d'}(Q')$. \end{lemma}

\begin{proof}
It suffices to show that the additional action of $\mathbb{C}^*$ can be compensated by the action of $G_{\bf d}(\tilde{Q})$. Indeed, given $\lambda\in \mathbb{C}^*$, we define $g\in G_{\bf d}(\tilde{Q})$ by $g_i=\lambda^{-1}{\rm id}_{V_i}$ for all $i\in \qt$. Then
$$(\lambda,{\rm id})\cdot((V_\alpha)_\alpha,(v_{i,k})_{i,k}),(\varphi_{i,l})_{i,l})=((V_\alpha)_{\alpha},(\lambda^{-1}\cdot v_{i,k})_{i,k},(\lambda\cdot\varphi_{i,l})_{i,l}))$$
and also $$(1,(g_i)_i)\cdot((V_\alpha)_\alpha,(v_{i,k})_{i,k}),(\varphi_{i,l})_{i,l})=((V_\alpha)_{\alpha},(\lambda^{-1}\cdot v_{i,k})_{i,k},(\lambda\cdot\varphi_{i,l})_{i,l})).$$
\end{proof}

We have thus proved the equivalence of the action of $G_{\bf d}(\tilde{Q})$ on $R_{\bf d}(Q)$ and of the action of $G_{\bf d'}(Q')$ on $R_{\bf d'}(Q')$. In particular, we have reinterpreted our problem of classifying representations of $Q$ for the restricted action to the problem of classifying representations of $Q'$ up to isomorphism. For the latter, we can now apply all known results on moduli spaces of quiver representations and translate them back to $Q$. 

\section{Moduli spaces for the trivial stability} \label{sec:trivialstability}

The proof of the following result is straightforward.
\begin{lemma} The quiver $Q'$ always has oriented cycles.
\end{lemma}

%\begin{proof}
%For every source $i$ in $\tilde{Q}$, we have $u_i\not=0$, thus an arrow from $\infty$ to $i$ in $Q'$. Similarly, for every sink $i$ in $\tilde{Q}$, we have $w_i\not=0$, thus an arrow from $i$ to $\infty$ in $Q'$. There exists at least one path from a source to a sink in $\tilde{Q}$.
%\end{proof}

Given this property, we don't have to introduce stabilty at this point, and we can look at the moduli space $M_{\bf d'}^{\rm ssimp}(Q')$ of semisimple representations of $Q'$. We apply the general theory by Le Bruyn-Procesi \cite{LeBruynProcesi90} for these spaces:

\begin{theorem} In the above situation, we have the following:
\begin{enumerate}
\item The moduli space $M_{\bf d'}^{\rm ssimp}(Q')$ is an affine irreducible variety.
\item It parametrizes the closed orbits of $G_{\bf d'}(Q')$ in $R_{\bf d'}(Q')$, thus the closed orbits of $G_{\bf d}(\tilde{Q})$ in $R_{\bf d}(Q)$.
\item It parametrizes the isomorphism classes of semisimple representations of $Q'$ of dimension vector ${\bf d'}$.
\item It has dimension $$\dim R_{\bf d}(Q)-\dim G_{\bf d}(\tilde{Q})$$
if there exists a simple representation of $Q'$ of dimension vector ${\bf d'}$.
\item It is a rational variety.
\end{enumerate}
\end{theorem}

\begin{proof}
Statement (1) is general, since the moduli space is defined as the ${\rm Spec}$ of the ring of invariants for a linear action of a reductive group on a vector space. The first part of Statement (2) is general, the second part follows from the observation in the previous section. Statement (3) is again general, the closed orbits in a representation variety correspond to the isoclasses of semisimple representations. Statement (4) follows from the following dimension calculation:
$$\dim M_{\bf d'}^{\rm ssimp}(Q')=1-\langle{\bf d'},{\bf d'}\rangle_{Q'}=\sum_{i\in \qt}(u_i+v_i)d_i-\langle{\bf\tilde d},{\bf\tilde d}\rangle_{\tilde{Q}}=\dim R_{\bf d}(Q)-\dim G_{\bf d}(\tilde{Q}).$$
Statement (5) follows immediately from a result of Schofield \cite{Schofield01} since the $\gcd$ of the entries of ${\bf d'}$ equals one.
\end{proof}

We now turn to coordinates for the moduli space $M_{\bf d'}^{\rm ssimp}(Q)$: again by general theory, they are given by traces along oriented cycles. Since $\tilde{Q}$ is acyclic, the only possible oriented cycles (wlog starting and ending in $\infty$) are given by $\gamma_{j,l}\omega\beta_{i,k}$ for a path $\omega:i\leadsto j$ from $i$ to $j$ in $\tilde{Q}$, an index $k=1,\ldots,u_i$, and an index $l=1,\ldots,w_i$. Thus, coordinates for the moduli space $M_{\bf d'}^{\rm ssimp}(Q')$ are provided by the functions $T_{\omega,k,l}$ mapping a representation $((V_\alpha),(v_{i,k}),(\varphi_{i,l}))$ to the scalar $\varphi_{j,l}V_\omega v_{i,k}$. This can be expressed as follows:

\begin{theorem}\label{t33} The moduli space $M_{\bf d'}^{\rm ssimp}(Q')$ is isomorphic to the image of the map
$$R_{\bf d}(Q)\simeq R_{\bf\tilde{d}}(\tilde{Q})\times\bigoplus_{i\in \qt}{\rm Hom}(U_i,V_i)\times\bigoplus_{i\in \qt}{\rm Hom}(V_i,W_i)\rightarrow\bigoplus_{\omega:i\leadsto j}{\rm Hom}(U_i,W_j)$$
given by
$$((V_\alpha),(f_i),(h_i))\mapsto (h_jV_\omega f_i)_{\omega:i\leadsto j}.$$
In other words, this map separates the closed orbits, equivalently, the isomorphism classes of semisimple representations. In particular, it separates isomorphism classes of simple representations.
\end{theorem}

{\bf Example:} Let $Q$ be a quiver of extended Dynkin type ${\tilde{D}_4}$
%\begin{figure}
%\includegraphics[width=0.5\textwidth, angle=0]{quiver.png}
%\end{figure}
for which we consider thin representations, that is, representations of dimension vector ${\bf d}=(1,1,1,1,1)$. We use the indexing
$$\begin{array}{ccccc}1&&&&4\\
&\searrow&&\nearrow&\\
&&3&&\\
&\nearrow&&\searrow&\\
2&&&&5\end{array}$$
of the vertices of $\tilde{Q}$ and consider $${\bf u}=(2,2,0,0,1),\; {\bf w}=(0,0,0,2,2).$$
A representation is thus given by the following datum:
$$\begin{array}{ccccccccc} U_1&\stackrel{\varphi}{\rightarrow}&\mathbb{C}&&&&\mathbb{C}&\stackrel{v}{\rightarrow}&W_4\\
&&&a\searrow&&\nearrow c&&&\\
&&&&\mathbb{C}&&&&\\
&&&b\nearrow&&\searrow d&&&\\
U_2&\stackrel{\psi}{\rightarrow}&\mathbb{C}&&&&\mathbb{C}&\stackrel{w}{\rightarrow}&W5\\
&&&&&&\uparrow\lambda&&\\
&&&&&&U_5&&\end{array}$$
for scalars $a,b,c,d,\lambda\in\mathbb{C}$, vectors $v,w\in\mathbb{C}^2$ and covectors $\varphi,\psi\in(\mathbb{C}^2)^*$. The above quotient map has target space
$${\rm Hom}(U_1,W_4)\oplus{\rm Hom}(U_1,W_5)\oplus{\rm Hom}(U_2,W_4)\oplus{\rm Hom}(U_2,W_5)\oplus{\rm Hom}(U_5,W_5),$$
which we can consider as embedded into
$${\rm Hom}(U_1\oplus U_2\oplus U_5,W_4\oplus W_5),$$
a space of $4\times 5$-matrices. The map itself assigns to the above datum the matrix
$$\left[\begin{array}{ccc}ac\cdot v\varphi&bc\cdot v\psi&0\\ ad\cdot w\varphi&bd\cdot w\psi&\lambda\cdot w\end{array}\right].$$

\section{Stability conditions}

We first comment on an alternative interpretation of ``double framed'' which, however, gives a different moduli space. In the above situation, define a quiver $Q''$ with set of vertices
$$Q''_0=\qt \cup\{0,\infty\},$$ and set of arrows
$$Q''_1=\tilde{Q}_1\cup\{\beta_{i,k}:0\rightarrow i\, |\, i\in \qt,\, k=1,\ldots,u_i\}\cup$$
$$\cup\{\gamma_{i,l}:i\rightarrow \infty\,|\, i\in \qt,\, l=1,\ldots,w_i\}.$$
We define a dimension vector ${\bf d''}$ for $Q''$ by
$$d''_i=d_i\mbox{ for }i\in \qt,\; d''_0=1,\; d''_\infty=1.$$
Then it is natural to consider representations which are $\Theta$-(semi-)stable for the stability function $\Theta_i=0\mbox{ for }i\in \qt,\; \Theta_0=1,\; \Theta_\infty=-1$. However, the expected (that is, in the presence of a stable representation) dimension of the moduli space $M_{\bf d''}^{\Theta-{\rm sst}}(Q'')$ is
$$\dim M_{\bf d''}^{\Theta-{\rm sst}}(Q'')=\dim R_{\bf d}(Q)-\dim G_{\bf d}(\tilde{Q})-1.$$
The reason is that the automorphism group of a stable representation of $\tilde{Q}$ consists just of the scalars, and these cannot compensate for separate dilations in the input arrows and the output arrows. The example of thin representations for a linearly oriented $A_3$ quiver shows this difference: If 
$$Q= i\rightarrow j\rightarrow k$$
then
$$Q'=\infty{\rightarrow\atop\leftarrow} j$$
whereas
$$Q''=0\rightarrow j\rightarrow\infty.$$
For $Q'$, we get an affine line as moduli space, whereas for $Q''$, the moduli space reduces to a single point.

Next, we make the condition of existence of simple representations for $Q'$ more explicit. First, it is easy to see from the definitions that the following holds:

\begin{lemma}\label{lemmasimple} A representation $((V_\alpha),(f_i),(h_i))$, where $(V_\alpha)$ is a representation of $\tilde{Q}$, defines a simple representation of $Q'$ if and only if the following holds:
\begin{itemize}
\item The largest $\tilde{Q}$-subrepresentation $V'\subset V$ such that $V'_i\subset{\rm Ker}(h_i)$ for all $i\in \qt$ is $0$.
\item The smallest $\tilde{Q}$-subrepresentation $V'\subset V$ such that ${\rm Im}(f_i)\subset V'_i$ for all $i\in \qt$ is $V$.
\end{itemize}
\end{lemma}

One can also use the numerical criterion due to Le Bruyn and Procesi \cite{LeBruynProcesi90} to effectively decide existence of a simple representation:

\begin{theorem} Assume that $Q'$ is not a type $\tilde{A}_n$ quiver (equivalently, $\tilde{Q}$ is not a linearly oriented type $A$ quiver with a single input at the source and a single output at the sink).Then there exists a simple representation of $Q'$ of dimension vector ${\bf d'}$ if and only if the following holds:
\begin{itemize}
\item There exists $i\in \qt$ such that $d_i(u_i+w_i)\not=0$,
\item $u_i\geq\langle{\bf d},e_i\rangle_{\tilde{Q}}$ for all $i\in \qt$,
\item $w_i\geq\langle e_i,{\bf d}\rangle_{\tilde{Q}}$ for all $i\in \qt$.
\end{itemize}

In case $Q'$ is of type $\tilde{A}_n$, there exists a simple representation if and only $d_i=1$ for all $i\in \qt$.
\end{theorem}

\begin{proof}
    The Le Bruyn-Procesi criterion states that there exists a simple representation if and only if the support of ${\bf d'}$ is strongly connected (meaning that there exists an oriented path between arbitrary vertices) and $\langle {\bf d'},e_i\rangle_{Q'},\langle e_i,{\bf d'}\rangle_{Q'}\leq 0$ for all $i\in Q'_0$. These easily translate to the conditions in the theorem.
\end{proof} 

In particular, these numerical criteria automatically hold if ${\bf\tilde{d}}$ is thin, resulting in the following.

\begin{corollary} Assume that $d_i=1$ for all $i\in \qt$. Then the moduli space is an affine irreducible rational toric variety of dimension $|Q_1|-|\qt|$.
\end{corollary}

\begin{proof}
    All statements are special cases of the previous results, except toricity. This is a general property of moduli spaces of thin representations, since the torus rescaling the arrows acts with a dense orbit (namely, all arrows being represented by non-zero scalars) on the quotient.
\end{proof}

%Concerning an explicit description of the moduli space, that is, defining relations between the coordinates, there is at least one obvious class of relations. Namely, for any path $\omega$ from $i$ to $j$, and any $k,k'=1,\ldots, u_i$ and $l,l'=1,\ldots, w_i$, we obviously have
%$$T_{\omega,k,l}T_{\omega,k',l'}=T_{\omega,k,l'}T_{\omega,k',l}.$$
%These relations can in fact be expressed more conceptually as follows:
%\begin{center}
%Any of the matrices $h_jV_\omega f_i$ has rank at most one.
%\end{center}
%This shows the advantage of the introduction of the vector spaces $U_i,W_i$ and of the notation $((V_\alpha),(f_i),(h_i))$ for representations of $Q'$.

We next introduce a variant of the moduli space $M_{\bf d'}^{\rm ssimp}(Q')$ which has the advantage of always being smooth. We introduce a stability function $\Theta$ for $Q'$ by
$$\Theta_\infty=\dim{\bf \tilde{d}},\; \Theta_i=-1\mbox{ for all }i\in \qt.$$
Here $\dim{\bf\tilde{d}}=\sum_{i\in \qt}d_i$, and thus $\Theta({\bf d'})=0$. We work with the convention (opposite to King's \cite{King94}) that a representation is semi-stable if $\Theta(U)\leq 0$ for all subrepresentations.

%In general, more relations betwen the coordinates $T_{\omega,k,l}$ can appear from overlapping paths in $\tilde{Q}$. One could hope that these can similarly be expressed in terms of rank conditions on a suitable matrix derived from the quotient map described in Theorem \ref{t33}, but this requires some more work.\\[3ex]

\begin{lemma}\label{lemmasemistable} A representation $((V_\alpha),(f_i),(h_i))$ is $\Theta$-semistable iff it is $\Theta$-stable iff the smallest $\tilde{Q}$-subrepresentation $V'\subset V$ such that ${\rm Im}(f_i)\subset V'_i$ for all $i\in \qt$ equals $V$.
\end{lemma}

The general theory of moduli spaces of (semi-)stable representations of quivers yields:

\begin{theorem} In the present situation, we have the following:
\begin{enumerate}
\item The moduli space $M_{\bf d'}^{\Theta-{\rm sst}}(Q')$ of $\Theta$-(semi-)stable representations of $Q'$ of dimension vector ${\bf d'}$ is a {\it smooth} irreducible quasiprojective variety.
\item It parametrizes the isomorphism classes of stable representations of $Q'$ of dimension vector ${\bf d'}$.
\item It has dimension $$\dim R_{\bf d}(Q)-\dim G_{\bf d}(\tilde{Q})$$
if there exists a $\Theta$-(semi-)stable representation of $Q'$ of dimension vector ${\bf d'}$.
\item It is a rational variety.
\item It admits a projective morphism $$M_{\bf d'}^{\Theta-{\rm sst}}(Q')\rightarrow M_{\bf d'}^{\rm ssimp}(Q').$$
\item It is toric in case ${\bf \tilde{d}}$ is thin.
\end{enumerate}
\end{theorem}

\begin{proof}
    Smoothness follows since stability and semi-stability coincide. All other properties are general.
\end{proof} 

So the advantage of this moduli space is its smoothness, and the fact that it separates all the  isoclasses of stable representations at the cost of having to work with quasiprojective varieties, that is, a mixture of affine and projective coordinates. %Whether this is an advantage to the previous moduli spaces, and whether it is worth the additional effort, depends on the specific aspects of neural networks which are to be modelled.

\section{Representation-theoretic description of the moduli spaces}

To shorten notation, we denote by $\mathcal{M}=M_{\bf d'}^{\rm ssimp}(Q')$ the moduli space parametrizing iso-classes of semisimple representations of $Q'$ of dimension vector ${\bf d'}$, let $\mathcal{M}'\subset\mathcal{M}$ be the open subset parametrizing simple representations, and let $\widetilde{\mathcal{M}}$ be the moduli space $M_{\bf d'}^{\Theta-{\rm sst}}(Q')$ parametrizing $\Theta$-stable representations. We give an explicit description of these spaces in more representation-theoretic terms. First we need to recall some notation. We refer the reader to \cite{Schiffler14} for the general concepts of representation theory.\\[2ex]
For a map $q\in{\rm Hom}_{\tilde{Q}}(V,W)$ between representations $V$, $W$ of $\tilde{Q}$, define its rank vector ${\bf r}(q)={\rm\bf dim}\,{\rm Im}(q)$ as the dimension vector (for $\tilde{Q}$) of its image. For a representation $V$ of $\tilde{Q}$ and a dimension vector ${\bf e}\leq{\rm\bf dim}\, V$, let ${\rm Gr}^{\bf e}(V)$ be the Grassmannian (see \cite{Cerulli}) of subrepresentations of dimension vector ${\rm\bf dim}\, V-{\bf e}$.\\[1ex]
For a vertex $i\in \qt$, let $P_i$ be the indecomposable projective representation associated to $i$, and let $I_i$ be the indecomposable injective representation associated to $i$. In the present setup, define
$$P=\bigoplus_{i\in \qt}P_i^{u_i},\; I=\bigoplus_{i\in \qt}I_i^{w_i}.$$
%We can now formulate the main result:

\begin{theorem}\label{thm:description-of-moduli-space} In the present situation, we have:
\begin{enumerate}
\item The moduli space $\mathcal{M}$ is isomorphic to the closed subset of the vector space ${\rm Hom}_{\tilde{Q}}(P,I)$ of maps $q$ with rank vector ${\bf r}(q)\leq{\bf\tilde{d}}$.
\item Under this isomorphism, the moduli space $\mathcal{M}'$ corresponds to the locally closed subset of maps $q\in{\rm Hom}_{\tilde{Q}}(P,I)$ with rank vector ${\bf r}(q)={\bf\tilde{d}}$.
\item The moduli space $\widetilde{\mathcal{M}}$ is isomorphic to the closed subset of the quasiprojective variety ${\rm Gr}^{\bf\tilde{d}}(P)\times{\rm Hom}_{\tilde{Q}}(P,I)$ of pairs $(U,q)$ such that $U\subset{\rm Ker}(q)$.
\item $\widetilde{\mathcal{M}}$ is a resolution of singularities of $\mathcal{M}$, which is an isomorphism over $\mathcal{M}'$, with all singular fibres being quiver Grassmannians for projective representations of $\tilde{Q}$.
\item $\widetilde{\mathcal{M}}$ is a vector bundle over ${\rm Gr}^{\bf\tilde{d}}(P)$.
\item Consequently, the Betti numbers of $\widetilde{\mathcal{M}}$ can be determined explicitely. In the thin case, $\widetilde{\mathcal{M}}$ admits an affine paving.
\end{enumerate}
\end{theorem}

\begin{proof}
Since
$${\rm Hom}_{\tilde{Q}}(P_i,V)\simeq V_i,\; {\rm Hom}_{\tilde{Q}}(V,I_i)\simeq V_i^*,$$
we can reinterpret Lemma \ref{l22} as follows. The points of the representation space $R_{\bf d}(Q)$ can be identified with triples $(V,f,h)$ consisting of a representation $V$ of $\tilde{Q}$ of dimension vector ${\bf \tilde{d}}$, a map $f:P\rightarrow V$, and a map $h:V\rightarrow I$. The action of $G_{\bf d}(\tilde{Q})$ is still given by the one described in Lemma \ref{l22}. From this description, it is obvious that
$$\pi:R_{\bf d}(Q)\rightarrow{\rm Hom}_{\tilde{Q}}(P,I),$$
$$(V,f,h)\mapsto h\circ f$$
is $G_{\bf d}(\tilde{Q})$-invariant. It is also clear that ${\bf r}(h\circ f)\leq{\rm\bf dim}\, V={\bf \tilde{d}}$, thus the image of $\pi$ is contained in the closed subset of maps with rank vector at most ${\bf\tilde{d}}$.

We first claim that this condition already describes the image completely. Namely, let $q\in{\rm Hom}_{\tilde{Q}}(P,I)$ with ${\bf r}(q)\leq{\rm\tilde{d}}$ be given. We consider its canonical factorization over its image
$$P\stackrel{f'}{\rightarrow}{\rm Im}(q)\stackrel{h'}{\rightarrow}I,$$
choose an arbitrary representation $C$ of dimension vector ${\bf\tilde{d}}-{\bf r}(q)$, define $V={\rm Im}(q)\oplus C$ (which has dimension vector ${\bf\tilde{d}}$), and define maps $f$ and $h$ by
$$f=f'\oplus 0:P\rightarrow V,\; h=h'\oplus 0:V\rightarrow I.$$
Then obviously $q=h\circ f$, thus $q$ is contained in the image of $\pi$.

Next we claim that $\pi$ is the quotient map for the $G_{\bf d}(\tilde{Q})$-action, that is, it separates closed orbits. Closed orbits corresponding to isoclasses of semisimple representations of $Q'$, we first describe these. First we note that the two conditions for simplicity of Lemma \ref{lemmasimple} directly translate into the map $f:P\rightarrow V$ being surjective, and the map $h:V\rightarrow I$ being injective. From this it follows easily that a triple $(V,f,h)$ has a closed orbit under $G_{\bf d}(\tilde{Q})$, equivalently corresponds to a semisimple representation of $Q'$, if and only if
\begin{itemize}
\item $V=W\oplus S$ for $S$ a semisimple representation of $\tilde{Q}$,
\item $f$ maps $P$ surjectively to $W$,
\item $h$ vanishes on $S$ and is injective on $W$.
\end{itemize}
In other words, the triple $P\stackrel{f}{\rightarrow}V\stackrel{h}{\rightarrow}I$ is of the form
$$P\stackrel{f'\oplus 0}{\rightarrow}W\oplus S\stackrel{h'\oplus 0}{\rightarrow}I$$
with $S$ semisimple, $f'$ surjective, and $h'$ injective.

Now we take again $q$ in the image of $\pi$. A closed orbit in its fibre is thus a triple as above such that $q=h\circ f$. But this forces $W$ to be ${\rm Im}(q)$, which is thus unique up to isomorphism, and forces $S$ to be the unique semisimple ($\tilde{Q}$ being acyclic) of dimension vector ${\bf\tilde{d}}-{\rm\bf dim}\, W$. This proves uniqueness of the closed orbit.

This proves part (1) of the theorem, and it also proves part (2) since the simple representations correspond precisely to the triples $(V,f,h)$ with $f$ surjective and $h$ injective, that is, ${\bf r}(h\circ f)={\bf\tilde{d}}$.

To prove part (3), we first note that Lemma \ref{lemmasemistable} translates to the following: a triple $(V,f,h)$ is $\Theta$-(semi-)stable iff $f$ is surjective. We can thus define a map from the semistable locus of $R_{\bf d}(Q)$ to ${\rm Gr}^{\bf\tilde{d}}(P)\times{\rm Hom}_{\tilde{Q}}(P,I)$ as follows:
$$\tilde{\pi}:(V,f,h)\mapsto ({\rm Ker}(f),h\circ f).$$
Obviously, ${\rm Ker}(f)\subset{\rm Ker}(h\circ f)$, thus the defining property of the subset defined in part (3) of the theorem is fulfilled. Again, $\tilde{\pi}$ is obviously $G_{\bf d}(\tilde{Q})$-invariant.

To prove that it is an isomorphism, take a pair $(U,q)$ such tat $U\subset{\rm Ker}(q)$. We then define $V=P/U$, define $f:P\rightarrow P/U=V$ to be the projection map, and define $h$ as the map $h:V=P/U\rightarrow I$ induced by $q:P\rightarrow I$. Then obviously ${\rm Ker}(f)=U$ and $h\circ f=q$, that is, $\tilde{\pi}(V,f,h)=(U,q)$. But it is also clear that this is the only possible choice (up to the $G_{\bf d}(\tilde{Q})$-action) given that $f$ has to be surjective. This proves that the fibre of any point in the image of $\tilde{\pi}$ consists of a unique orbit of a semistable representations, proving (3).

The first claim of part (4) is general. To analyse the fibres of the morphism $\widetilde{\mathcal{M}}\rightarrow\mathcal{M}$, we just have to fix $q$ and ask for the possible choices of $U\in{\rm Gr}^{\bf\tilde{d}}(P)$ such that $U\subset{\rm Ker}(q)$. But ${\rm Ker}(q)$ is again a projective representation of $\tilde{Q}$ (since the path algebra is hereditary), and thus the possible choices are encoded in the quiver Grassmannian ${\rm Gr}^{\bf e}({\rm Ker}(q))$ for a suitable ${\bf e}$.

For part (5), we fix $U\subset P$ and ask for the possible choices of $q$ such that $U\subset{\rm Ker}(q)$. But these are precisely encoded in ${\rm Hom}(P/U,I)$, which is a vector space whose dimension ($I$ being injective) does only depend on the dimension vector of $U$, that is, is constant along the quiver Grassmannian. This proves the vector bundle property.

Part (6) follows from part (5) once the claimed properties are true for ${\rm Gr}^{\bf\tilde{d}}(P)$, and this follows from \cite{Rei}. The theorem is proved.
    
\end{proof}

%But this follows e.g. from the methods of my paper ``Framed quiver moduli...''; I'm happy to work out the details if needed.

%{\color{blue} I would like to see how the Betti numbers change when we add to the quiver a layer, for example fully connected or convolutional. Also when we add more vertices to the layers. A small example could be a 3 fully connected layers quiver. }

We will now translate this description of the moduli spaces back to linear algebra. To do this, we will identify ${\rm Hom}(P,I)$ with $\bigoplus_{\omega:i\leadsto j}{\rm Hom}(U_i,W_j)$ (considered as the target of the quotient map earlier) in such a way that the rank conditions of the previous theorem can be made explicit.

%We first fix some notation.

Recall the space $\bigoplus_{\omega:i\leadsto j}{\rm Hom}(U_i,W_j)$ of tuples of linear maps $(q_\omega)_\omega$. For every vertex $i\in \qt$, we can consider the map formed by all paths passing through $i$:
$$q^{(i)}=\bigoplus_{\omega:j\leadsto i}\bigoplus_{\omega':i\leadsto k}q_{\omega'\omega}:\bigoplus_{\omega_j\leadsto i}U_j\rightarrow\bigoplus_{\omega':i\leadsto k}W_k.$$

For a vector space $V$ and $k\leq\dim V$, we denote by ${\rm Gr}^q(V)$ the Grassmannian of {\it $k$-codimensional} subspaces of $V$. For every arrow $\alpha:i\rightarrow i'$ in $\tilde{Q}$, we have a natural map
$$p_\alpha:\bigoplus_{\omega:j\leadsto i}U_j\rightarrow\bigoplus_{\omega:j\leadsto i'}U_j$$
which places an element $u_\omega\in U_j$ corresponding to a path $\omega:j\leadsto i$ in the component corresponding to the path $\alpha\omega$, that is,
$$p_\alpha((u_\omega)_\omega)=(\left\{\begin{array}{ccc}u_{\omega'}&,&\omega=\alpha\omega'\\ 0&,&\mbox{otherwise}\end{array}\right\})_\omega.$$

\begin{corollary}~ \begin{enumerate}
\item The moduli space $\mathcal{M}$ is isomorphic to the closed subset of $\bigoplus_{\omega:i\leadsto j}{\rm Hom}(U_i,W_j)$ of tuples of linear maps $(q_\omega)_\omega$ such that for all $i\in \qt$, the map
$q^{(i)}$ has rank at most $d_i$. 
\item Under this isomorphism, the moduli space $\mathcal{M}'$ corresponds to tuples where all $q^{(i)}$ have rank $d_i$.
\item The moduli space $\widetilde{\mathcal{M}}$ is isomorphic to the closed subset of
$$\prod_{i\in \qt}{\rm Gr}^{d_i}(\bigoplus_{\omega:j\leadsto i}U_j)\times\bigoplus_{\omega:i\leadsto j}{\rm Hom}(U_i,W_j)$$
of tuples $((V'_i)_i,(q_\omega)_\omega)$ such that the following holds:
\begin{enumerate}
\item for all arrows $\alpha:i\rightarrow i'$ in $\tilde{Q}$, we have $p_\alpha(V'_i)\subset V'_{i'}$,
\item for all $i\in \qt$, we have $V'_i\subset{\rm Ker}(q^{(i)})$.
\end{enumerate}
\end{enumerate}
\end{corollary}

\begin{proof}
Using the identification ${\rm Hom}_Q(P_i,V)\simeq V_i$ we see that ${\rm Hom}(P_i,I_j)$ has a basis indexed by the paths from i to j. Using the definition of $P$ and $I$, this gives an identification
$${\rm Hom}_Q(P,I)\simeq\bigoplus_{\omega:i\rightarrow j}{\rm Hom}(U_i,W_j).$$
The representation
$$(\bigoplus_{\omega:j\leadsto i}U_j)_i,(p_\alpha)_\alpha)$$
described above is precisely the representation $P$. Dually, one can make the representation $I$ explicit as
$$(\bigoplus_{\omega':i\leadsto k}U_k,(i_\alpha)_\alpha)$$
for suitable maps $i_\alpha$. Using these descriptions, we can make the above isomorphism explicit: an element
$$(q_\omega)_\omega\in\bigoplus_{\omega:i\leadsto j}{\rm Hom}(U_i,W_j)$$
corresponds to the map $q$ of quiver representations from $P$ to $I$ whose $i$-component is precisely the map $q^{(i)}$ defined above. Then the first two parts of the corollary follow from the above theorem. 

To prove the third part, we first note that condition (3)(a) of the theorem precisely means that the $V'_i$ form a subrepresentation of $P$, that is, $(V'_i)_i$ defines a point in the quiver Grassmannian ${\rm Gr}^{\bf\tilde{d}}(P)$. Using the previous identification, the condition that the subrepresentation $(V'_i)_i$ is contained in the kernel of $q$ translates to (3)(b). This proves the theorem.
    
\end{proof}

We illustrate this description in the running example. The representation $P$ is then given by
$$\begin{array}{ccccc}U_1&&&&U_1\oplus U_2\\
&\searrow&&\nearrow&\\
&&U_1\oplus U_2&&\\
&\nearrow&&\searrow&\\
U_2&&&&U_1\oplus U_2\oplus U_5\end{array}$$
and the representation $I$ is given by
$$\begin{array}{ccccc}W_4\oplus W_5&&&&W_4\\
&\searrow&&\nearrow&\\
&&W_4\oplus W_5&&\\
&\nearrow&&\searrow&\\
W_4\oplus W_5&&&&W_5,\end{array}$$
where the arrows are represented by the obvious inclusion and projection maps.

The space $\bigoplus_{\omega:i\leadsto j}{\rm Hom}(U_i,W_j)$ is
$${\rm Hom}(U_1,W_4)\oplus{\rm Hom}(U_1,W_5)\oplus{\rm Hom}(U_2,W_4)\oplus{\rm Hom}(U_2,W_5)\oplus{\rm Hom}(U_5,W_5),$$
whose elements can be depicted as 
$$\begin{array}{cccc}&U_1&U_2&U_5\\W_4&A&B&0\\ W_5&C&D&E.\end{array}
$$
The corresponding map $q:P\rightarrow I$ is given by the following components $q^{(i)}$:
$$\begin{array}{ccccc}\left[\begin{array}{cc}A\\ C\end{array}\right]&&&&\left[\begin{array}{cc}A&B\end{array}\right]\\
&&&&\\
&&\left[\begin{array}{cc}A&B\\ C&D\end{array}\right]&&\\
&&&&\\
\left[\begin{array}{c}B\\ D\end{array}\right]&&&&\left[\begin{array}{ccc}C&D&E\end{array}\right].\end{array}$$
Consequently (since ${\bf\tilde{d}}=(1,1,1,1,1)$ in the running example), the moduli space $\mathcal{M}$ is isomorphic to the space of $4\times 5$-matrics
$$\left[\begin{array}{ccc}A&B&0\\ C&D&E\end{array}\right],$$
where $A$, $B$, $C$ and $D$ are $2\times 2$-matrices and $E$ is a vector of length $2$,
such that the above submatrices have rank at most $1$, and $\mathcal{M}'$ is the open subset where they all have rank $1$.

For the moduli space $\widetilde{\mathcal{M}}$, we consider $5$-tuples of subspaces
$$V'_1\subset U_1,\; V'_2\subset U_2,\; V'_3\subset U_1\oplus U_2,\; V'_4\subset U_1\oplus U_2,\; V'_5\subset U_1\oplus U_2\oplus U_5$$
which are subject to the conditions
$$V'_1,V'_2\subset V'_3=V'_4\subset V'_5$$
and
$$\begin{array}{ccccc}V'_1\subset{\rm Ker}(\left[\begin{array}{cc}A\\ C\end{array}\right])&&&&V'_4\subset{\rm Ker}(\left[\begin{array}{cc}A&B\end{array}\right])\\
&&&&\\
&&V'_3\subset{\rm Ker}(\left[\begin{array}{cc}A&B\\ C&D\end{array}\right])&&\\
&&&&\\
V'_2\subset{\rm Ker}(\left[\begin{array}{c}B\\ D\end{array}\right])&&&&V'_5\subset{\rm Ker}(\left[\begin{array}{ccc}C&D&E\end{array}\right]).\end{array}$$
In our situation, all $V'_i$ are $1$-codimensional subspaces, thus hyperplanes, and the identification
$${\rm Gr}^1(V)\simeq\mathbb{P}(V^*)$$
allows us to make the following identifications: we identify $V'_5$ with the kernel of a form $$[\psi_1,\psi_2,\psi_5]\in(U_1\oplus U_2\oplus U_5)^*,$$ and the above inclusion conditions between the $V'_i$ give
$$V'_1={\rm Ker}(\psi_1),\; V'_2={\rm Ker}(\psi_2),\; V'_3=V'_4={\rm Ker}([\psi_1,\psi_2]).$$
The inclusions of the $V'_i$ into kernels can then be translated into the conditions that the following block matrices have rank at most one:
$$\begin{array}{ccccc}\left[\begin{array}{cc}\psi_1\\A\\ C\end{array}\right]&&&&\left[\begin{array}{cc}\psi_1&\psi_2\\ A&B\end{array}\right]\\
&&&&\\
&&\left[\begin{array}{cc}\psi_1&\psi_2\\A&B\\ C&D\end{array}\right]&&\\
&&&&\\
\left[\begin{array}{c}\psi_2\\ B\\ D\end{array}\right]&&&&\left[\begin{array}{ccc}\psi_1&\psi_2&\psi_5\\ C&D&E\end{array}\right].\end{array}$$
Consequently, the moduli space $\widetilde{M}$ is the closed subset of
$$\mathbb{P}(U_1^*)\times\mathbb{P}(U_2^*)\times\mathbb{P}(U_5^*)\times$$
$$\times {\rm Hom}(U_1,W_4)\oplus{\rm Hom}(U_1,W_5)\oplus{\rm Hom}(U_2,W_4)\oplus{\rm Hom}(U_2,W_5)\oplus{\rm Hom}(U_5,W_5)$$
of tuples
$$(\psi_1,\psi_2,\psi_5,A,B,C,D,E)$$
fulfilling the above rank conditions, which are in fact just quadratic conditions on the entries of the vectors, covectors and matrices.

\section{Categorical point of view}\label{sec:monoidal}
%\subsection{Monoidal category}

The  construction of Theorem \ref{thm:description-of-moduli-space} suggests to study simple objects in the following category: Let $P$ be a fixed projective of $\tilde{Q}$, and $I$ a fixed injective of $\tilde{Q}$. We define \[{\mathcal R}(P,I) = \{ (V,f,h) \; | \; V \in \rep \tilde{Q},  P\stackrel{f}{\rightarrow} V\stackrel{h}{\rightarrow}I \}\]
to be the category  whose objects are given by morphisms $q= h \circ f$ from $P$ to $I$ factoring through some  representation $V$ of $\tilde{Q}$. A morphism from $P\stackrel{f}{\rightarrow} V\stackrel{h}{\rightarrow}I$ to $P\stackrel{f'}{\rightarrow} V'\stackrel{h'}{\rightarrow}I$ is given by a morphism $g: V \to V'$ such that the following diagram commutes:

\[
\xy
(-16,0)*+{P}="0";
(3,0)*+{}="1";
(0,8)*+{V}="2";
(0,-8)*+{V'}="4";
(-3,0)*+{}="5";
(16,0)*+{I}="6";
{\ar^{f} "0";"2"};
{\ar^{h} "2";"6"};
{\ar_{f'} "0";"4"};
{\ar_{h'} "4";"6"};
{\ar^{g} "2";"4"};
{\ar@{}|{} "0";"1"};
{\ar@{}|{} "5";"6"};
\endxy
\]

It is clear from the discussion of the proof of Theorem \ref{thm:description-of-moduli-space} that the isomorphism classes of objects in the category ${\mathcal R}(P,I)$ can be identified with the set of $G_{\bf d}(\tilde{Q})$-orbits on the representation space  $R_{\bf d}(Q)$.
The category ${\mathcal R}(P,I)$ is not (pre-)additive since the space of morphisms is a non-linear subspace of  $\Hom_{\tilde{Q}}(V,V')$, but we still have the following:

\begin{lemma}
A morphism $g$ in ${\mathcal R}(P,I)$ is a monomorphism precisely when $g$ is injective, and an epimorphism in ${\mathcal R}(P,I)$ precisely when $g$ is surjective.
\end{lemma}

\begin{definition}
We say that an object $(V,f,h)$ in ${\mathcal R}(P,I)$ is {\em simple} if there is no proper monomorphism $i: (V',f',h') \to (V,f,h)$ and no proper epimorphism $p: (V,f,h) \to (V'',f'',h'')$.
\end{definition} 
For any $(V,f,h)$ in $\mathcal R$, we have the monomorphism 

\[
\xy
(-16,0)*+{P}="0";
(3,0)*+{}="1";
(0,8)*+{V}="2";
(0,-8)*+{\Im f}="4";
(-3,0)*+{}="5";
(16,0)*+{I}="6";
{\ar^{f} "0";"2"};
{\ar^{h} "2";"6"};
{\ar_{f} "0";"4"};
{\ar_{h} "4";"6"};
{\ar^{incl} "4";"2"};
{\ar@{}|{} "0";"1"};
{\ar@{}|{} "5";"6"};
\endxy
\]
and the epimorphism 
\[
\xy
(-16,0)*+{P}="0";
(3,0)*+{}="1";
(0,8)*+{V}="2";
(0,-8)*+{V/\Ker \; h}="4";
(-3,0)*+{}="5";
(16,0)*+{I}="6";
{\ar^{f} "0";"2"};
{\ar^{h} "2";"6"};
{\ar^{f'} "0";"4"};
{\ar^{h'} "4";"6"};
{\ar^{proj} "2";"4"};
{\ar@{}|{} "0";"1"};
{\ar@{}|{} "5";"6"};
\endxy
\]
The proof of Theorem \ref{thm:description-of-moduli-space} therefore shows that an object $(V,f,h)$ in ${\mathcal R}(P,I)$ is simple presicely when $f$ is surjective and $h$ is injective, which coincides with the notion of simple objects in the representation space  $R_{\bf d}(Q)$.
Moreover, the semisimple elements in the representation space  $R_{\bf d}(Q)$  correspond to objects of the form 
$(W \oplus S,f \oplus 0,h \oplus 0)$ with $S$ a semisimple representation of $\tilde{Q}$, $f$ surjective to $W$ and $h$ injective.

\begin{remark}
We do not see a natural notion of direct sum in the category ${\mathcal R}(P,I)$. Not even the object $P \to 0 \to I$ embeds into every object, since the diagram
\[
\xy
(-16,0)*+{P}="0";
(3,0)*+{}="1";
(0,8)*+{V}="2";
(0,-8)*+{0}="4";
(-3,0)*+{}="5";
(16,0)*+{I}="6";
{\ar^{f} "0";"2"};
{\ar^{h} "2";"6"};
{\ar_{0} "0";"4"};
{\ar_{0} "4";"6"};
{\ar^{incl} "4";"2"};
{\ar@{}|{} "0";"1"};
{\ar@{}|{} "5";"6"};
\endxy
\]
does not commute.
\end{remark}

However, the category ${\mathcal R}(P,I)$ does have the structure of a {\em monoidal} category \cite{Benabou, MacLane97}.

%(the axioms were first explicitly formulated by B\'enabou \cite{Benabou} around 1963, who called them ``cat\'egories avec multiplication" and by Mac Lane \cite{MacLane}, who called them categories with multiplications; the renaming is due to Eilenberg\cite{Eileberg}):eilenber monoi

\begin{definition}\cite{MacLane97} A \emph{(strict) monoidal category} $({\bf C}, \otimes, E)$ is a category $\bf{C}$ with a bifunctor $\otimes: \bf{C}\times \bf{C} \rightarrow \bf{C}$ which is associative $$ \otimes (\otimes \times \bf{1}) = \otimes (\bf{1}\times \otimes): {\bf C}\times {\bf C}\times {\bf C}\rightarrow \bf{C}$$
and with an object $E$ which is a left and right unit for $\otimes$
$$ \otimes (E\times {\bf 1}) = id_{C} = \otimes ({\bf 1}\times E), $$
where $ E\times {\bf 1}$ is the functor ${\bf C}\rightarrow {\bf C}\times {\bf C}; c\mapsto (E, c)$.
A monoidal category $({\bf C}, \otimes, E)$ is \emph{symmetric} if for every pair of objects $A$ and $B$ in $\bf{C}$ there is an isomorphism $$S_{A,B}:A\otimes B\rightarrow B\otimes A$$ that is \emph{natural} in both $A$ and $B$
\end{definition}

Note that the associative law, in the definition above, states that the binary operation $\otimes$ is associative both for objects and for arrows. Similarly, the unit law above means that $E\otimes c = c = c\otimes E$ for objects $c$ and $1_{E}\otimes f = f = f\otimes 1_{E}$ for arrows $f$.
\bigskip

For representations $V,V'$ of $\tilde{Q}$, we denote by $V \otimes V'$ the {\em pointwise} tensor product as defined in \cite{Hers08}. It is obtained on vertices $v$ as
    \[
        (V \otimes V')_v = V_v \otimes_\bC V'_v,
    \]
and for every arrow $\alpha \in \tilde{Q}$ one obtains the tensor product of matrices
    \[
        (V \otimes V')_\alpha = V_\alpha \otimes_\bC V'_\alpha \; .
    \]
 
%{\color{purple} Note: Certainly Herschend knows what he does, but why can't we take the tensor product $V \otimes_S V'$ over the semisimple ring $S = k\tilde{Q}/ \mbox{rad} k\tilde{Q}$? }

%{\color{blue} I agree with you, it seems we CAN consider the semisimple ring S. But what would be the advantage of this?}

Moreover, for  morphisms $f: P \to V $ and $f': P \to V' $, we denote by $f \otimes f': P \to V \otimes V' $ the morphism induced by the maps 
$f_i\otimes f'_i : P_i \to V_i \otimes V'_i , p \mapsto f_i(p) \otimes f'_i(p),$ likewise for morphisms $h: V \to I.$
This yields an associative bifunctor $\otimes$ on ${\mathcal R}(P,I)$:

\[ (V,f,h) \otimes (V',f',h') := (V \otimes V',f \otimes f',h \otimes h').
\]

Consider the case where every vector space $V$ on the representation of $Q$ has dimension one. Denote by $E$ the object of ${\mathcal R}(P,I)$ which corresponds to the representation of $Q$ with all arrows being represented by the identity map.
Then $E$ is a left and right unit for $\otimes$, and we get: 
\begin{proposition} The tensor product $\otimes$ and the unit $E$ turns ${\mathcal R}(P,I)$ into a symmetric monoidal category.
\end{proposition}
\begin{proof}
We are using the constructions of $\otimes$ and $E$ as defined in the discussion above.
It is shown in \cite[Theorem 2]{Hers08}
that the category of representations of any quiver is a monoidal category with the pointwise tensor product and the identity  $E$ given by the representation with all arrows being represented by the identity map. We obtain the same properties for the category ${\mathcal R}(P,I)$.
\end{proof}

\begin{definition}
The {\em Picard group Pic}$(C, \otimes, E)$ of a symmetric monoidal category is the abelian group defined as the group of isomorphism classes of objects which are invertible with respect to the tensor product.
\end{definition}

\begin{proposition}
Assume that the dimension $d_s$ of all sinks and all sources $s$ of $Q$ is one. Then the Picard group  Pic$({\mathcal R}(P,I),\otimes, E)$ is formed by the isomorphism classes of thin representations  of $Q$ all of whose arrows being represented by non-zero scalars. We can therefore identify Pic$({\mathcal R}(P,I),\otimes, E)$ with a dense open subset of $\mathcal M$.
\end{proposition}
\begin{proof}
If  $(V,f,h)$ is invertible under tensor product, there is an element $(V',f',h')$ in $({\mathcal R}(P,I),\otimes, E)$ such that $(V,f,h) \otimes (V',f',h') \cong E.$ This in particular implies $V_v \otimes_\bC V'_v \cong \bC,$ thus $V$ must be thin.
Moreover, we have for every arrow $\alpha$ that $(V \otimes_\bC V')_\alpha =  V_\alpha V'_\alpha = 1$, thus all scalars $V_\alpha$ are invertible.
Every such element $(V,f,h)$ is certainly stable, and since we consider isomorphism classes of objects in the Picard group, we obtain an embedding of Pic$({\mathcal R}(P,I),\otimes, E)$ onto a dense open subset of $\mathcal M$.
\end{proof}

\section{Applications to Neural Networks} \label{sec:applications}

\subsection{The network function}

Let us recall some definitions from \cite{AJ20} on neural networks.

\begin{definition} \cite{AJ20} \label{neural-net}
    Let $Q$ be a connected finite acyclic quiver without multiple arrows. A \textbf{neural network} over $Q$ is a pair $(W,f)$ where $W$ is a thin representation of $Q$ and $f=(f_q)_{q \in \widetilde{Q}}$ are activation functions, i.e., almost everywhere differentiable functions $f_q: \mathbb{C} \to \mathbb{C}$.
\end{definition}

We assume there are $p=d+d'$ source vertices, from which $d$ of them are called \textbf{input vertices} and the remaining $d'$ are called \textbf{bias vertices} and $q$ sinks (also called \textbf{output vertices}) in the quiver $Q$. Elements $x \in \mathbb{C}^d$ are called \textbf{input vectors} for the neural network. We define the \textbf{activation output} at vertex $v\in \qt$ of the neural network $(W,f)$ with respect to the input $x$ as
\[
    \act(W,f)_v(x)= \left\{ \begin{array}{ll}
        x_v & \text{ if } v \text{ is an input vertex,}  \\
        1 & \text{ if } v \text{ is a bias vertex,} \\
        f_v \left( \displaystyle\sum_{\alpha \in \zeta_v} W_\alpha \act(W,f)_{s(\alpha)}(x) \right) & \text{ in any other case.}
    \end{array} \right.
\]
Here $\zeta_v$ is the set of arrows of $Q$ with target $v$. We also define the \textbf{pre-activation} at vertex $v$ of the neural network $(W,f)$ with respect to the input $x$ as
\[
    \pre\act(W,f)_v(x)= \left\{ \begin{array}{ll}
        1 & \text{ if } v \text{ is a source vertex,} \\
       \displaystyle\sum_{\alpha \in \zeta_v} W_\alpha \act(W,f)_{s(\alpha)}(x)  & \text{ in any other case.}
    \end{array} \right.
\]

%Similarly, $\pre\act(W,f)_v(x) = \act(W,f)_v(x)$ if $v$ is an input vertex and $\pre\act(W,f)_v(x) = \displaystyle\sum_{\alpha \in \zeta_v} W_\alpha \act(W,f)_{s(\alpha)}(x)$ if not.

Denote by $Q_t = \{ t_1,...,t_k \}$ the set of sink vertices of the quiver $Q$. The \textbf{network function} of the neural network $(W,f)$ is the map $\Psi(W,f): \mathbb{C}^d \to \mathbb{C}^q$ given by 
\[
    \Psi(W,f)(x) = \Big( \act(W,f)_{t_1}(x), ...,  \act(W,f)_{t_k}(x)\Big) \in \mathbb{C}^q.
\]

%$\Psi(W,f)(x) = \left( \act(W,f)_v(x) \right)_{v \text{ is an output vertex}}$.

The \textbf{knowledge map} of the neural network $(W,f)$ is
\[
\begin{array}{clcl}
    \varphi(W,f) : & \mathbb{C}^d & \to & thin(Q) \\ %_d \mathcal{R}(Q)_k  \\
     & x & \mapsto & W_x^f
\end{array}
\]
where
\[
    \left( W_x^f \right)_{\alpha} = \left\{ \begin{array}{ll}
        W_\alpha x_{s(\alpha)}  & \text{ if } s(\alpha) \text{ is an input vertex,}  \\ \\
        
        W_\alpha & \text{ if } s(\alpha) \text{ is a bias vertex,} \\ \\
        
        W_\alpha \dfrac{\act(W,f)_{s(\alpha)}(x)}{\pre\act(W,f)_{s(\alpha)}(x)} & \text{ if } s(\alpha) \text{ is a hidden vertex.} \\
    \end{array} \right.
\]
Note that the knowledge map is not well-defined when $\pre\act(W,f)_{s(\alpha)}(x)$ is zero for some $\alpha$ but this corresponds to a set of measure zero and so we ignore these cases. We also define the map
\[
    \begin{array}{clcc}
        \Hat{\Psi}: & \mathcal{M}(Q) & \to & \mathbb{C}^q  \\
         & [V] & \mapsto & \Psi(V,1)(1,...,1)
    \end{array}
\]
where $(V,1)$ denotes a neural network whose activation functions are all equal to the identity map (equivalently, activations and pre-activations of vertices are always equal) and $(1,...,1) \in \mathbb{C}^d$

\begin{theorem} \cite{AJ20} \label{thm:AJ}
Let $(W,f)$ be a neural network and $x\in \mathbb{C}^d$ an input vector. The following diagram is commutative
    \[
        \begin{tikzcd}
            \mathbb{C}^d \arrow[rrrr, "\Psi(W f)"] \arrow[drr, "\varphi(W f)"] & & & &  \mathbb{C}^q  \\
            & & \mathcal{M} \arrow[urr, "\Hat{\Psi}"] & &\\
        \end{tikzcd}
    \]
\end{theorem}

This factorization of the network function can easily be extended in the following way. Let $Map(X,Y)$ be the category whose objects are maps $f:X \to Y$ and morphisms are pairs of maps $(\alpha_1,\alpha_2):f \to g$ making the obvious diagram commutative. Denote by $NN(Q)$ the grupoid of neural networks.

\begin{theorem} \label{thm:func-fact}
    The factorization of the network function through the knowledge map is functorial in the neural network, that is, the following diagram is commutative
    \[
        \begin{tikzcd}
            NN(Q) \arrow[rr, "\Psi"] \arrow[dr, "\varphi"] & & Map(\mathbb{C}^d,\ \mathbb{C}^q).  \\
            & Map(\mathbb{C}^d, \ \mathcal{M}) \arrow[ur, "\Hat{\Psi}^*"] & \\
        \end{tikzcd}
    \]
\end{theorem}

\begin{proof}
    It follows from the invariance of the network function under the action of $G$ defined in \cite{AJ20} that $\Psi$ is a well-defined functor. For a morphism of neural networks $\tau:(W,f) \to (V,g)$ define $(1,\tau):\varphi(W,f) \to \varphi(V,g)$. This makes $\varphi$ into a functor. Commutativity of the diagram is proved analogously as that of the previous theorem.
\end{proof}

Therefore, the evaluation of the network function of a neural network $(W,f)$ on a fixed input $x$ is equivalent to evaluating the network function of the neural network $(W_x^f,1)$ on the input $(1,...,1)$. So we will study here the network function of neural networks of the form $(V,1)$. We want to describe this network function on the level of moduli spaces first for any dimension vector not necessarily thin.

%{\color{blue} In the moduli space we can (and should) assume there are no activation functions. Because the knowledge map $\varphi(W,f)$ takes an input vector $x$ to a quiver representation without activation functions. Note also that the output of the network on a given $x$ can be obtained by computing the output on the induced quiver representation $W_x^f$ with respect to the canonical input vector $(1,...,1)$.}

The network function for a representation $V$ of dimension vector ${\bf d}$ of $Q$ is thus defined inductively by propagating values as follows. We are given a tuple $(x_1,\ldots,x_p)\in\bigoplus_{k=1}^pV_{s_k}$. We define $x_i$ for all $i\in Q_0$ inductively by
\begin{itemize}
\item $x_{s_k}=x_k$ for $k=1,\ldots,p$,
\item in case $i$ is not a source in $Q$, and all $x_j$ for $j$ a vertex admitting an arrow to $i$ already defined, we define $$x_i=\sum_{\alpha:j\rightarrow i}V_\alpha(x_j).$$
\end{itemize}

We will denote $N_V := \Psi(V,1)$ for ease of notation. The network function $N_V$ is therefore defined as
$$N_V:\bigoplus_{k=1}^pV_{s_k}\rightarrow\bigoplus_{l=1}^qV_{t_l},$$
$$N_V(x_1,\ldots,x_p)=(x_{t_1},\ldots,x_{t_q}).$$

%One can show that $N_V$ only depends on the orbit of $V$ under $G_{\bf d}(\tilde{Q})$. 

We will now show that $N_V$ only depends on the point in the moduli space $\mathcal{M}$ associated to $V$ and identify it using the explicit description of $\mathcal{M}$.

Recall that
$$U_i=\bigoplus_{\alpha:s_k\rightarrow i}V_{s_k},\; W_i=\bigoplus_{\alpha:i\rightarrow t_l}V_{t_l}.$$

We thus have natural maps
$${\rm in}:\bigoplus_{k=1}^pV_{s_k}\rightarrow\bigoplus_{i\in \qt}U_i,$$
$${\rm out}:\bigoplus_{i\in \qt}W_i\rightarrow \bigoplus_{l=1}^qV_{t_q}$$
collecting the natural injection (resp.~projection) maps. Also recall that the moduli space $\mathcal{M}$ is given as the image of the map
$$\pi:R_{\bf d}(Q)\rightarrow\bigoplus_{\omega:i\leadsto j}{\rm Hom}(U_i,W_j),$$
$$\pi(V)=(h_jV_\omega f_i)_{\omega: i\leadsto j}.$$

\begin{lemma} We have
$$N_V={\rm out}\circ\pi(V)\circ{\rm in},$$
thus, in particular, $N_V$ only depends on the point $\pi(V)$ associated to $V$ in the moduli space $\mathcal{M}$, and $N_V$ depends only linearly on $\pi(V)$.
\end{lemma}

\begin{proof}
 From the recursive definition of $N_V$, we deduce that
$$x_{t_l}=\sum_{k=1}^p\sum_{\omega:s_k\leadsto t_l}V_\omega(x_{s_k}).$$
Now the function ${\rm in}$ first distributes the $x_s$ to all the spaces $U_i$, and $f_i$ collects all the maps from $V_\alpha$ from sources to $\tilde{Q}$. Conversely, $h_i$ collects all the maps from $\tilde{Q}$ to sinks, and ${\rm out}$ distributes from all the spaces $W_j$ to the $V_{t_l}$. This implies that $N_V$ is given as the composition of ${\rm in}$, the matrix $(h_jV_\omega f_i)_{\omega:i\leadsto j}$, and the map ${\rm out}$. This proves the lemma.
\end{proof}

Again we illustrate this in the running example. The representation

$$\begin{array}{ccccccccc} U_1&\stackrel{\varphi}{\rightarrow}&\mathbb{C}&&&&\mathbb{C}&\stackrel{v}{\rightarrow}&W_4\\
&&&a\searrow&&\nearrow c&&&\\
&&&&\mathbb{C}&&&&\\
&&&b\nearrow&&\searrow d&&&\\
U_2&\stackrel{\psi}{\rightarrow}&\mathbb{C}&&&&\mathbb{C}&\stackrel{w}{\rightarrow}&W5\\
&&&&&&\uparrow\lambda&&\\
&&&&&&U_5&&\end{array}$$
corresponds to the following representation of the original quiver:

$$\begin{array}{ccccccccc} \bullet &\stackrel{}{\rightarrow}&\bullet&&&&\bullet&\stackrel{}{\rightarrow}&\bullet\\
&&&\searrow&&\nearrow &&&\\
&&&&\bullet&&&&\\
&&&\nearrow&&\searrow &&&\\
\bullet&\stackrel{}{\rightarrow}&\bullet&&&&\bullet&\stackrel{}{\rightarrow}&\bullet\\
&&&&&&\uparrow&&\\
&&&&&&\bullet&&\end{array}$$

%\begin{center}
%\includegraphics[width=0.7\textwidth, angle=180]{IMG_2002.JPG}
%\end{center}

and the network function works as follows. We are given scalars $x_1,x_2,x_3$ and attach the following values to the vertices inductively:
\begin{itemize}
\item $x_{s_1}=x_1$, $x_{s_2}=x_2$, $x_{s_3}=x_3$,
\item $x_1=\varphi_1x_1+\varphi_2x_2$,
\item $x_2=\psi_1x_1+\psi_2x_2$,
\item $x_3=a\varphi_1x_1+a\varphi_2x_2+b\psi_1x_1+b\psi_2x_2$,
\item $x_4=ac\varphi_1x_1+ac\varphi_2x_2+bc\psi_1x_1+bc\psi_2x_2$,
\item $x_5=ad\varphi_1x_1+ad\varphi_2x_2+bd\psi_1x_1+bd\psi_2x_2+\lambda x_3$,
\item $x_{t_1}=acv_1\varphi_1x_1+acv_1\varphi_2x_2+bcv_1\psi_1x_1+bcv_1\psi_2x_2+
adw_1\varphi_1x_1+adw_1\varphi_2x_2+bdw_1\psi_1x_1+bdw_1\psi_2x_2+\lambda w_1 x_3$. 
Thus,

$$x_{t_1}=(acv_1\varphi_1+bcv_1\psi_1+adw_1\varphi_1+bdw_1\psi_1)x_1+(acv_1\varphi_2+bcv_1\psi_2+adw_1\varphi_2+bdw_1\psi_2)x_2+\lambda w_1x_3,$$
\item $x_{t_2}=acv_2\varphi_1x_1+acv_2\varphi_2x_2+bcv_2\psi_1x_1+bcv_2\psi_2x_2+
adw_2\varphi_1x_1+adw_2\varphi_2x_2+bdw_2\psi_1x_1+bdw_2\psi_2x_2+\lambda w_2 x_3$,

and then:

$$x_{t_2}=acv_2\varphi_1+bcv_2\psi_1+adw_2\varphi_1+bdw_2\psi_1)x_1+(acv_2\varphi_2+bcv_2\psi_2+adw_2\varphi_2+bdw_2\psi_2)x_2+\lambda w_2x_3.$$
\end{itemize}
The maps ${\rm in}, {\rm out}$ are given by
$${\rm in}(x_1,x_2,x_3)=(\left[x_1\atop x_2\right],\left[x_1\atop x_2\right],x_3),$$
$${\rm out}(\left[y_1\atop y_2\right],\left[z_1\atop z_2\right])=(y_1+z_1,y_2+z_2).$$
We calculate ${\rm out}\circ\pi(V)\circ{\rm in}$, 
$${\rm out}\circ\pi(V)\circ{\rm in}=\left[\begin{array}{cc}e_1^*&e_1^*\\ e_2^*&e_2^*\end{array}\right]\circ\left[\begin{array}{ccc}acv\varphi&bcv\psi&0\\ adw\varphi&bdw\psi&\lambda w\end{array}\right]\circ\left[\begin{array}{ccc}e_1&e_2&0\\ e_1&e_2&0\\ 0&0&1\end{array}\right]=$$
$$=\left[\begin{array}{ccc}acv_1\varphi_1+bcv_1\psi_1+adw_1\varphi_1+bdw_1\psi_1&acv_1\varphi_2+
bcv_1\psi_2+adw_1\varphi_2+bdw_1\psi_2&\lambda w_1\\
acv_2\varphi_1+bcv_2\psi_1+adw_2\varphi_1+bdw_2\psi_1&acv_2\varphi_2+
bcv_2\psi_2+adw_2\varphi_2+bdw_2\psi_2&\lambda w_2
\end{array}\right],$$
and that is indeed the matrix representing the map $(x_1,x_2,x_3)\mapsto(x_{t_1},x_{t_2}).$

\subsection{The network category}
Let us focus know in the case of thin representations. Here, we will introduce a category that contains the image of the knowledge map $\varphi(W,f)$ of a neural network $(W,f)$, that we call the {\em network category } $\mathcal{N}_Q$ of the underlying quiver $Q$.
\begin{definition}Let $\mathcal{N}_Q$ be the (not necessarily full) subcategory of $\rep Q$ formed by all representations $W$ of $Q$ with $W_v = \bC$ for all vertices $v$ of $Q,$ and whose morphism set from $W$ to $W'$ is defined to be all morphisms $g$ in $\rep Q$ from $W$ to $W'$, given by a family of complex numbers $(g_v)_{v \in Q_0}$, such that $g_v = 1 $ whenever $v$ is a sink or a source.
\end{definition}
This condition we use is closed under composition, so we do get an actual category. Denote by ${\bf d}$ the thin dimension vector $d_i = 1$ for all vertices $i$ of $Q$. Then we have the following result:
\begin{lemma}
The isomorphism classes of objects in the category $\mathcal{N}_Q$  correspond bijectively to the $G_{\bf d}(\tilde{Q})$-orbits in $R_{\bf d}(Q)$.
\end{lemma}

The results of the previous sections allow to identify $\mathcal{N}_Q$ with the full subcategory of
${\mathcal R}(P,I)$ given by all triples $(V,f,h)$ where $V$ is a thin representation of $\tilde{Q}$, and $P$ and $I$ correspond to the choice of one-dimensional spaces for all sinks and sources of $Q$. Recall from \cite{MacLane97} the category theoretic definition of a groupoid:

\begin{definition}\cite{MacLane97}
    A \emph{groupoid} is a small category in which every morphism is an isomorphism.
\end{definition}
\begin{proposition} The full subcategory $\mathcal{N}^S_Q$ of $\mathcal{N}_Q$ formed by simple objects is a groupoid.
  \end{proposition}
\begin{proof}
 $\mathcal{N}^S_Q$ is certainly a small category. We need to show that every morphism is an isomorphism.
 This follows from Schur's Lemma applied to objects in rep $Q'$. The simples in $\mathcal{N}_Q$ correspond to simple representations of $Q'$, as explained in Section \ref{section:deframing}.
 The Schur Lemma states that for simple representations $U$ and $V$of the quiver $Q'$, every morphism $\phi: U \to V$ is either zero or invertible. But the zero morphism is not a morphism in the subcategory $\mathcal{N}_Q$, therefore every morphism in  $\mathcal{N}^S_Q$ must be an isomorphism.
\end{proof}

Note that the category $\mathcal{N}_Q$ itself is {\em not} a groupoid in general. 
For example, consider the thin representations
\[ U : \quad \bC \stackrel{1}{\to} \bC \stackrel{0}{\to} \bC 
\]
and \[ V : \quad \bC \stackrel{0}{\to} \bC \stackrel{1}{\to} \bC 
\]
where $Q$ is  the linearly oriented quiver of type $A_3$.
Then there is a morphism $\phi$ in $\mathcal{N}_Q$ from $U$ to $V$ as follows:
$$\begin{array}{ccccccc}
U & : & \bC & \stackrel{1}{\to} & \bC & \stackrel{0}{\to}& \bC \\
\downarrow^\phi & & { ^1}\downarrow \; & & ^0\downarrow \; & & \downarrow^1\\
V & : & \bC & \stackrel{0}{\to} & \bC & \stackrel{1}{\to} & \bC 
\end{array}$$
Of course $\phi$ is not invertible, but also $U$ and $V$ are not simple (viewed as objects in ${\mathcal R}(P,I)$). Proposition 2 of \cite{Hers08} implies the following:

%\subsection{The monoidal category of thin representations}
%  \label{sec:TenProd}

\begin{corollary}\label{tensor product}

The tensor product of thin representations is a thin representation.
    
\end{corollary}

%\begin{proof}
%    Let $W$ and $V$ be thin representations of an acyclic quiver $Q$. Then we have
%    \[
%        (W \otimes_\bC V)_\alpha = W_\alpha \otimes_\bC V_\alpha = W_\alpha V_\alpha,
%    \]
%    for every $\alpha \in Q_1$, and
%    \[
%        (W \otimes_\bC V)_v = W_v \otimes_\bC V_v = \bC \otimes_\bC \bC = \bC,
%    \]
%    for every $v\in Q_0$.
%\end{proof}

\begin{corollary}
    Tensor product of stable thin representations is a stable thin representation, and tensor product of thin representations is non-stable if at least one of the two representations is non-stable.
\end{corollary}
\begin{proof}
It follows from \ref{tensor product}.
\end{proof}

\begin{proposition}
  $(\mathcal{N}_Q, \otimes, E)$ is a (strict) symmetric monoidal category. 
\end{proposition}
\begin{proof}
$\mathcal{N}_Q$ is a full subcategory  of the monoidal ${\mathcal R}(P,I)$, closed under the operation $\otimes$.
\end{proof}

\begin{remark}
    The classes in $Pic(\mathcal{N}_Q, \otimes, E)$ are given by objects in $\mathcal{N}^S_Q$, but in general $\mathcal{N}^S_Q$ contains more objects. A representation of $Q'$ might be simple while having an arrow being represented by the scalar $0$.
\end{remark} 

\subsection{ReLU neural networks} 

In the previous subsections, we were ignoring the activation functions on the neural networks by the argument given by theorem \ref{thm:func-fact}, where each input data gets mapped to a quiver representation and so evaluating the network function on a fix input is equivalent to evaluating a vector on a neural network of the form $(V,1)$. Here, we study moduli spaces for neural networks with a specific activation function for which we can find invariance under the action of a subgroup of $G$. That is, as opposed to the two previous subsections, here we consider a moduli space on which the neural network $(W,f)$ lives together with the image of the knowledge map $\varphi(W,f)$ for a specific activation function $f$. This activation function is the so-called ReLU function (Rectified Linear Units) which is defined as follows:
\[
    ReLU(x) = max(0,x).
\]
The network functions of neural networks with ReLU activation functions are clearly invariant under the restricted action of the group
\[
    G^+ = \prod_{i\in \tilde{Q}} \mathbb{R}^+
\]
when considered as a subgroup of $G$. In this case, we will now describe an approach to defining moduli spaces of ReLU neural networks via the symplectic reduction approach to quiver moduli of \cite{King94}.

We choose Hermitian forms on all complex vector spaces $V_i$ for $i\in\tilde{Q}_0$ (note that this results in Hermitian forms on all $U_i$ and $W_i$). This choice determines a unitary subgroup $U(V_i)\subset{\rm GL}(V_i)$, and we consider the maximal compact subgroup $U_{\bf d}=\prod_{i\in \tilde{Q}_0}U(V_i)\subset G_{\bf d}(\tilde{Q})$. 
We consider the momentum map
$$\mu_{{\bf d}}:R_{{\bf d}}(Q)\rightarrow \bigoplus_{i\in \tilde{Q}_0}{\rm End}(V_i),$$
$$\mu((V_\alpha),(f_i),(h_i))\mapsto(\sum_{\alpha:\rightarrow i}V_\alpha V_\alpha^*-\sum_{\alpha:i\rightarrow}V_\alpha^*V_\alpha+f_if_i^*-h_i^*h_i)_i.$$
 Applying the general results of \cite[Section 6]{King94} to the deframed quiver $Q'$, we easily find:
 
\begin{theorem} We have canonical homeomorphisms
$$\mathcal{M}\sim\mu^{-1}(0)/U_{\bf d},$$
$$\widetilde{\mathcal{M}}\sim\mu^{-1}(({\rm id}_{V_i})_i)/U_{\bf d}.$$
\end{theorem}

In fact, we can immediately adopt this to the real setup: assuming the $V_i$ to be real vector spaces and choosing scalar products, we have orthogonal subgroups $O(V_i)\subset {\rm GL}(V_i)$, determining a maximal compact subgroup $O_{\bf d}\subset G_{\bf d}(\tilde{Q})$. Again defining the map $\mu_{\bf d}$ as above, we find homeomorphisms 
$\mathcal{M}\sim\mu^{-1}(0)/O_{\bf d}$ and $\widetilde{\mathcal{M}}\sim\mu^{-1}(({\rm id}_{V_i})_i)/O_{\bf d}$ for the real versions of the moduli spaces.

Specializing to thin representations, that is, $d_i=1$ for all $i\in \tilde{Q}_0$, we are in the following situation:

All $V_i$ are just one-dimensional vector spaces $\mathbb{R}$, all maps $V_\alpha$ reduce to scalars, all $f_i$ are linear forms on the spaces $U_i$,  and all $h_i$ are vectors in $W_i$. The structure group is $G_{\bf d}=(\mathbb{R}^*)^{Q_0}$,  with maximal compact subgroup $(O_1(\mathbb{R}))^{Q_0}=(\pm 1)^{Q_0}$. The previous theorem now reads:

\begin{corollary}
The moduli space $\mathcal{M}$ is homeomorphic to $\mu_{{\bf 1}}^{-1}(0)/(\pm 1)^{Q_0}$, and the moduli space $\widetilde{\mathcal{M}}$ is homeomorphic to $\mu_{{\bf 1}}^{-1}((1)_i)/(\pm 1)^{Q_0}$. Explicitly, we have
$$\mu_{{\bf 1}}^{-1}(0)=\{((V_\alpha,(f_i),(h_i))\, |\sum_{\alpha:\rightarrow i}V_\alpha^2-\sum_{\alpha:i\rightarrow}V_\alpha^2+||f_i||^2-||h_i||^2=0\},$$
$$\mu_{{\bf 1}}^{-1}((1)_i)=\{((V_\alpha,(f_i),(h_i))\, |\sum_{\alpha:\rightarrow i}V_\alpha^2-\sum_{\alpha:i\rightarrow}V_\alpha^2+||f_i||^2-||h_i||^2=1\}.$$
\end{corollary}

This symplectic description of the moduli space in the thin real case is the key to defining a moduli space for the smaller structure group $G^+$. Namely, the observation that the maximal compact subgroup of $\mathbb{R}^*$ is $\pm 1$, whereas the maximal compact subgroup of $\mathbb{R}^+$ is trivial, motivates the following:

\begin{definition} We define moduli spaces for the $(\mathbb{R}^+)^{Q_0}$-action on $R_{{\bf 1}}(Q)$ as
$$\mathcal{M}^+=\mu_{{\bf 1}}^{-1}(0),\; \widetilde{\mathcal{M}}^+=\mu_{{\bf 1}}^{-1}((1)_i).$$
\end{definition}

Since the moduli space $\mathcal{M}$ arises from $\mathcal{M}^+$ by factoring by a finite group action, we find that we have a canonical map $\mathcal{M}^+\rightarrow\mathcal{M}$, which is a finite branched covering.

We define the network function $W_{(V,f,h)}:\bigoplus_{i\in Q_0}U_i\rightarrow\bigoplus_{i\in Q_0}W_i$ using the activation functions ${\rm ReLU}(x)=\max(x,0)$ as follows: we start with input data $(u_i\in U_i)_{i\in Q_0}$,  recursively define
$$v_i={\rm ReLU}\Big(f_i(u_i)+\sum_{\alpha:j\rightarrow i}V_\alpha(v_j)\Big),$$
and finally obtain the output data $w_i=h_i(v_i)$, for all $i\in Q_0$. 

Our main observation, immediately resulting from the $\mathbb{R}^+$-invariance of ReLU and our definition of the moduli spaces, is:
\begin{remark}The network map $W:R_{{\bf 1}}(Q)\rightarrow{\rm Fun}(\bigoplus_{i\in Q_0}U_i\rightarrow\bigoplus_{i\in Q_0}W_i)$ is $(\mathbb{R}^+)^{Q_0}$-invariant,   thus descends to a map $$W:\left\{
{\widetilde{\mathcal{M}}^+}
\atop
{\mathcal{M}^+}
\right\} \rightarrow{\rm Fun}(\bigoplus_{i\in Q_0}U_i\rightarrow\bigoplus_{i\in Q_0}W_i).$$ \end{remark}

We illustrate this observation in the simplest example of the quiver $Q$ being a single vertex, and $U_i=\mathbb{R}=W_i$. \\  
A double-framed thin representation is then just a pair of scalars $(f,h)\in\mathbb{R}^2$,   subject to the base change action $g\cdot(f,h)=(g\cdot f,\frac{1}{g}\cdot h)$. The network functions $u\mapsto h\cdot\max(f\cdot u,0)$.    We then find
$$\widetilde{\mathcal{M}}=\mathcal{M}=\mathbb{R}\mbox{ via }(f,h)\mapsto h\cdot f.$$  
$$\mathcal{M}^+=\{(f,h)\, |\, f^2-h^2=0\},\;\widetilde{\mathcal{M}}^+=\{(f,h)\, |\, f^2-h^2=1\}.$$

\subsection*{Acknowledgements} We acknowledge that some results of sections 2, 3 and 4 have been obtained independently by Alexander Schmitt with a more geometric approach that will appear in an upcoming paper.

\bibliographystyle{plain}

\bibliography{references}

\newpage

\appendix

\section{Moduli spaces and neural networks}

In this appendix, we will present a motivation to the study of moduli spaces of double framed quiver representations by analysing concepts coming from deep learning and artificial neural networks.

For this, we first present the machine learning concepts needed to introduce the reader to the back-propagation algorithm (used to compute the gradient of a neural network) in order to write it down in terms of the approach to neural networks given by quiver representations as introduced in \cite{AJ20} and then show that the back-propagation algorithm also factorizes through the moduli space of double framed thin quiver representations.

In deep learning (i.e., machine learning with neural networks), the weights (the thin representation $W$ as introduced in Definition \ref{neural-net}) of a neural network are obtained by optimization (usually gradient descent or some variant of it) over a dataset while leaving the activation functions unchanged. During this process, the network function is called several times to compute a loss value with its output, with which a gradient is computed using the back-propagation algorithm. The information needed to perform this \textit{training} algorithm consists of the following two objects:

\begin{itemize}
    \item The network function.
    \item The gradient with respect to a subset of input data.
\end{itemize}

It was shown in \cite{AJ20} that the network function factorizes through the knowledge map $\varphi(W,f)$ of the neural network $(W,f)$. We will show in this appendix that the back-propagation algorithm also factorizes through the knowledge map as a motivation to the study of moduli spaces of double framed quiver representations. In other words, we show that the knowledge map and the moduli space of the underlying quiver of a neural network has information from both the network function and the gradient of the neural network and therefore this can clearly have important implications in the study of neural networks.

\subsection{Machine learning concepts}

We will present here the main concepts from machine learning in a more algebraic way but still equivalent to their meaning in machine learning.

A \textbf{labeled data set} is a finite set $D = \{ (x_i,y_i) \ : \ x_i \in \mathbb{C}^d, \ y_i \in\mathbb{C}^q \  i=0,...,n\}$. The machine learning practitioners might want to enforce the \textit{labels} $y_i$ to take discrete values, for example, $y_i\in \{ 0,1,...,9 \}$ for a classification problem with $10$ different classes. But this is only one way of representing labels of a data set or the outputs of a \textit{classifier} (a function that takes an input $x$ to a discrete decision $y$). This is represented by pre-assigning a label to each sink of the network quiver (one for each class of the classification problem) and using the convention that the sink with higher value when evaluating $\Psi(W,f)(x)$ corresponds to the \textit{desicion} of the neural network (as a classifier). This is effectively what happens even if one uses an activation function after the network function, like a softmax:
\[
\sigma(z) = \left( \frac{e^{z_1}}{ \sum_i e^{z_i}} , \cdots,\frac{e^{z_q}}{ \sum_i e^{z_i}}   \right).
\]
After applying the softmax activation function to $\Psi(W,f)(x)$, the values obtained are interpreted as \textit{the confidence} of the neural network on $x$ being of each corresponding class. Usually, a data set is thought as a finite set of samples taken from a hypothetical function $F:\mathbb{C}^d \to \mathbb{C}^q$. The choice of a data set (or more generally, a \textit{task}) determines the number of input and output vertices of any quiver over which we will take neural networks to train them. Therefore, the data determines the double framing of the hidden quiver.

In practice, one works over $\mathbb{R}$ instead of $\mathbb{C}$, but this can easily be resolved by restricting to real values. A \textbf{loss function} is a differentiable function $\mathcal{L}:\mathbb{R}^q \times \mathbb{R}^q \to \mathbb{R}$, that is used to measure how well is the network function computing outputs that match the given labels (some authors call the composition $\mathcal{L} \circ \Psi(W,f)$ the loss function \cite{Goodfellow-et-al-2016}). One simple example of loss function is the mean-square error $\mathcal{L}(z,z')=\sum_{i=1}^q (z_i-z_i')^2$, and specifically for classification tasks one chooses the cross-entropy loss function, which assumes that the entry vectors $z$ and $y$ are probability vectors (for example, after applying the softmax function), given by
%\[
%    \mathcal{L}(z,y)=\left\{ \begin{array}{ll}
%        -log(z) & \text{ if } y=1  \\
%        -log(1-z) & \text{ if not.} 
%    \end{array} \right.
%\]
\[
    \mathcal{L}(z, y) = - \sum_{i} z_i log(y_i).
\]
%Here we assume
%Note that we use $y$ for the second entry of this loss function to be consisted with our data notation, as our labels $y_i$ are probability vectors indicating higher probability at the entry where it is equal to 1. 

Data sets are really big, for example, the famous MNIST dataset \cite{MNIST} of hand written digits $n=60,000$, which are labeled hand written images, is one of the simplest datasets used to train neural networks. So the training has to be made efficient, and for this the data set gets randomly shuffled and partitioned into batches (or mini-batches) of sizes 8, 16, 32, 64, 128, 512 or 1024 but certainly not higher because of computational issues. One takes a batch of data $x_1,...,x_t$ and computes $w=\frac{1}{n}\sum_{i=1}^n \Psi(W,f)(x_i)$ which is used to compute $\sigma(w)$. Then the gradient of $\mathcal{L} \circ \Psi(W,f)$ is computed using $\sigma(w)$, the representation $W$ and the activation and pre-activations of every vertex on the neural network.

As when we compute the network function $\Psi(W,f)(x)$, we need to keep track of some values (like the pre-activations and activation outputs of vertices, defined after Definition \ref{neural-net}), we will need to keep track of some other values along the computation of the gradient with the back-propagation algorithm. 

We will write down here the algorithm used to compute the gradient of the composition $\mathcal{L} \circ \Psi(W,f)$, and for this we will assume the vertices of the hidden layers are all fully connected layers (each vertex in a layer is connected to every every vertex in the next layer), so that the network function $\Psi(W,f)$ is a composition of linear maps followed by coordinate wise activation functions:
\[
    \Psi(W,f) = W^{[L]} \circ f^{[L]} \circ \cdots \circ W^{[2]} \circ f^{[2]} \circ W^{[1]}.
\]

\begin{remark}
    In deep learning, researchers refer to functions defined as above as neural networks. Combinatorially, this is equivalent to the network function as defined in \cite{AJ20}. Note also that we restrict here to Multi Layer Perceptrons (every layer is fully connected to the next) only to illustrate how the algorithm works. Combinatorially, this algorithm is clearly analogous for any other network quiver including convolutional layers (each $W^{[i]}$ is a circulant matrix), pooling layers, batch norm layers, etc., see \cite{AJ20}.
\end{remark}

\subsection{Back-propagation}

Consider a neural network $(W,f)$ over a network quiver $Q$ given only by fully-connected layers together with a labeled dataset $D$ partitioned into mini-batches. The quiver is formed by $L$ fully-connected hidden layers where the $\ell$-th layer has vertices $Q^{[\ell]}_0$ and point-wise activation functions $f^{[\ell]}$. The weights of arrows starting on the $(\ell-1)$-th layer and ending on the $\ell$-th layer can be arranged into a matrix defining a linear map $W^{[\ell]}:\mathbb{C}^{n_{\ell-1}} \to \mathbb{C}^{n_\ell}$, and then the network function of $(W,f)$ can be written as
\[
    \Psi(W,f) = W^{[L]} \circ f^{[L-1]} \circ W^{[L-1]} \circ \cdots \circ f^{[1]} \circ W^{[1]}.
\]
Observe that by adding bias vertices at every layer, the linear maps $W^{[\ell]}$ become affine transformations. The ``\textit{basic operations of neural networks}'' are affine transformations followed by point-wise non-linear activation functions.

The back propagation algorithm (originally introduced in \cite{Rumelhart86}, see chapter 6 of \cite{Bengio13} for a modern approach) computes the gradient $dW$ with respect to each arrow of the quiver, and so is given by matrices $dW^{[1]},...,dW^{[L]}$ just as $W$. 
\begin{remark}
    The back-propagation algorithm computes the gradient of any \textit{computational graph}, which is effectively a quiver representation.
\end{remark}
Denote $a^{[\ell]} = \Big( \act(W,f)_v(x) \Big)_{v \in Q_0^{[\ell]}} \in \mathbb{C}^{n_\ell}$, and $z^{[\ell]} = \Big( \pre \act(W,f)_v(x) \Big)_{v \in Q_0^{[\ell]}} \in \mathbb{C}^{n_\ell}$. We will keep track of the following values $da^{[L]},...,da^{[1]}$ in that order (and therefore the name back-propagation). We set
\[
    da^{[L]} = \dfrac{\partial \mathcal{L} }{\partial a^{[L]}} \Big( a^{[L]},y \Big).
\]
Then, as a matrix, we define
\[
    dW^{[L]} = da^{[L]} \Big( a^{[L-1]} \Big)^T,
\]
and also define
\[
    da^{[L-1]} = \Big( W^{[L]} \Big)^T da^{[L]}.
\]

From here, one defines inductively
\[
    dW^{[\ell]} := \Big( da^{[\ell]} \odot f^{[\ell]}\big( z^{[\ell]} \big) \Big) \Big( a^{[\ell-1]} \Big)^T,
\]
and
\[
    da^{[\ell-1]} := \Big( W^{[\ell]} \Big)^T \Big( da^{[\ell]} \odot df^{[\ell]} \big( z^{[\ell]} \big) \Big),
\]
where $\odot$ is the Hadamard (point-wise) product of vectors of the same size. Note that, by definition, $dW$ is a matrix multiplication obtained from the pre-activations of vertices, activation outputs of vertices, the weights, the activation function and its derivative $df$. 
%The layers of $dW$ are computed from the output layer to the input layer by multiplying with the activations outputs of vertices, and then {\color{blue}(I claim that)} they are representations of the opposite quiver of $Q$.

\begin{remark}
    Note that we are abusing the notation by writing $dW$ for the gradient although it also depends on the chosen input $x$ and the activation functions $f$.
\end{remark}

\subsection{Combinatorial back-propagation}

We now write the previous algorithm in a combinatorial way. First on the computation of $\Psi(W,f)(x)$ one keeps track of the activation outputs at the $\ell$-th layer:
\[
    a^{[\ell]} = \Big( \act(W,f)_v(x) \Big)_{v \in Q_0^{[\ell]}} \in \mathbb{C}^{n_\ell},
\]
and pre-activations
\[
    z^{[\ell]} = \Big( \pre \act(W,f)_v(x) \Big)_{v \in Q_0^{[\ell]}} = \Hat{\Psi}\left( |W_x^f|_\ell \right)  \in \mathbb{C}^{n_\ell},
\]
where $|W_x^f|_\ell$ is the image of $W_x^f$ under the forgetful functor that forgets layers $\ell+1,...,L$. We will abuse the notation by writing $ \Hat{\Psi}(W_x^f)_v$ for the output at vertex $v$ when we feed the vector $(1,...,1)$ to $W_x^f$.  For the backward computation we define for each vertex $v\in Q_0$:
\[
    \big( da \big)_{v} = \left\{ \begin{array}{ll}
        {\displaystyle\sum_{\alpha:t(\alpha)=v}}  \dfrac{\partial \mathcal{L}}{\partial t(\alpha)} \left( \Hat{\Psi}(W_x^f)_{t(\alpha)} , y \right) & \text{ if } v \text{ is a sink,} \\ \\
        {\displaystyle\sum_{\alpha:s(\alpha)=v}} W_\alpha (da)_{t(\alpha)} & \text{ if } v \text{ is one arrow away from a sink}, \\ \\
        {\displaystyle\sum_{\alpha:s(\alpha)=v}} W_\alpha \Big( \big(da\big)_{t(\alpha)} \ f_{t(\alpha)}\big( \Hat{\Psi}(W_x^f)_{t(\alpha)} \big) \Big) & \text{ in any other case.}
    \end{array} \right.
\]
Therefore, the gradient is given by
\[
    \big( dW \big)_\alpha = \left\{ \begin{array}{ll}
        \dfrac{\partial\mathcal{L}}{\partial t(\alpha)} \Big( \Hat{\Psi}(W_x^f)_{t(\alpha)},y \Big) \Big( f_{s(\alpha)} \left( \Hat{\Psi}(W_x^f)_{s(\alpha)} \right) \Big) & \text{ if } t(\alpha) \text{ is a sink,} \\ \\
        \big( da \big)_{t(\alpha)} df_{t(\alpha)}\big( \Hat{\Psi}(W_x^f)_{t(\alpha)} \big) f_{s(\alpha)}\big( \Hat{\Psi}(W_x^f)_{s(\alpha)} \big) & \text{ in any other case.}
    \end{array} \right.
\]
The back-propagation algorithm is therefore a map
\[
    \begin{array}{clcl}
        b(W,f) : & \mathbb{C}^p & \to &  R_{\bf d}(Q) \\
         & x & \mapsto & b(W,f)(x) := dW.
    \end{array}
\]
We can see that $dW$ is computed from the thin representation $W_x^f \in R_{\bf d}(Q)$, so in principle we can define a map that depends on thin representations $V \in R_{\bf d}(Q)$ and produces a thin representation in $R_{\bf d}(Q)$ by changing $V$ for $W_x^f$ in the formulas above. This defines a map
\[
    \begin{array}{clcl}
        d(W,f) : & R_{\bf d}(Q) & \to &  R_{\bf d}(Q). \\
    \end{array}
\]
We have proved the following:
\begin{theorem}
    The following diagram is commutative:
    \[
        \begin{tikzcd}
            \mathbb{C}^d \arrow[rrrr, "b(W f)"] \arrow[drr, "\varphi(W f)"] & & & & R_{\bf d}(Q).  \\
            & & \mathcal{M} \arrow[urr, "d(W f)"] & & \\
        \end{tikzcd}
    \]
\end{theorem}

Therefore, the knowledge map $\varphi(W,f)$ and the moduli space encode all the essential information of neural networks: the network function and the back-propagation algorithm.
%\subsection{The gradient and the moduli space}
We argue, however, that back-propagation $b(W,f)$ should take values in $R_{\bf d}(Q^{op})$, because of the following straightforward result.

\begin{proposition}
  Let $\tau:(W,f) \to (V,g)$ be an isomorphism of neural networks, then for any $x$ the isomorphism $\tau$ induces isomorphisms of thin representations $\tau:W_x^f \to V_x^f$ and $\tau:dW \to dV$, where we consider $dW$ and $dV$ as representations over the opposite quiver $Q^{op}$.
\end{proposition}

\subsection{Final discussion}

As can be appreciated, every building block of neural networks (weighted graph used to compute an input to output function and an algorithm to compute the gradient with respect to inputs and a loss function) can be written down in purely representation theoretic terms. In deep learning \cite{Goodfellow-et-al-2016, Bengio13, Hinton07}, the idea of studying ``\textit{internal representations}'' (i.e., activation outputs of vertices on hidden layers) comes from the paradigm of trying to understand neural networks by the information carried by the neurons. Although it is true that the output is produced by propagating information through all the layers of the network, what one uses is only the input to output function of the neural network (the network function). For instance, applying an isomorphism of double-framed quiver representations to a neural network preserves the network function but it changes the ``\textit{neuron values}'' (i.e., the activation outputs of hidden vertices) making any analysis of neural networks in terms of the neuron values not well defined. Finally, as noted by M. Tegmark \cite{Tegmark} on the lack of quantum effects in brain processes, the key difference lies not in the neurons that carry this information but in the patterns were by they're connected. Quiver representations can capture precisely this phenomenon.

\end{document}